\documentclass[11pt,a4paper,reqno]{amsart}
\usepackage[latin1]{inputenc}
\usepackage{amsmath,hyperref}
\usepackage{amssymb} 
\usepackage[a4paper, left=3cm, right=3cm, bottom=2cm]{geometry}
\usepackage{amsfonts, bm, verbatim,tikz}
\usepackage[all,cmtip]{xy}
\usepackage{bbold}
\usepackage{cleveref}

\theoremstyle{plain}
\newtheorem{theorem}{Theorem}[section]
\newtheorem{corollary}[theorem]{Corollary}
\newtheorem{lemma}[theorem]{Lemma}

\newtheorem{prop}[theorem]{Proposition}
\newtheorem{Scholium}[theorem]{Scholium}

\theoremstyle{definition}
\newtheorem{defn}[theorem]{Definition}
\newtheorem{example}[theorem]{Example}

\theoremstyle{remark}

\newcommand{\F}{\mathbb F}

\newcommand{\nc}{\newcommand}

\nc{\thm}{\theorem}
\nc{\cor}{\corollary}
\nc{\mc}{\mathcal}
\nc{\mb}{\mathbb}
\nc{\mf}{\mathfrak}
\nc{\ul}{\underline}
\nc{\ol}{\overline}
\nc{\N}{\mb N}
\nc{\R}{\mb R}
\nc{\Z}{\mb Z}
\nc{\Q}{\mb Q}

\nc{\gm}{\gamma}
\nc{\Hom}{Hom}
\nc{\SO}{SO}
\nc{\Spin}{Spin}
\nc{\lip}{\langle}
\nc{\rip}{\rangle}
\nc{\eps}{\varepsilon}
\nc{\la}{\lambda}
\nc{\vt}{\vartheta}
\nc{\half}{\frac{1}{2}}
\nc{\beq}{\begin{equation*}}
\nc{\eeq}{\end{equation*}}

\DeclareMathOperator{\sgn}{sgn}

\DeclareMathOperator{\Sq}{Sq}

\nc{\dmo}{\DeclareMathOperator}
\dmo{\ccup}{cup}
\dmo{\HH}{H}
\dmo{\Sym}{Sym}
\dmo{\HHom}{Hom}
\dmo{\Res}{Res}
\dmo{\im}{im}
\dmo{\Perm}{Perm}
\dmo{\rank}{rank}
\dmo{\GL}{GL}
\dmo{\st}{st}
\dmo{\ord}{ord}
\dmo{\Syl}{Syl}
\dmo{\topp}{top}
\dmo{\Irr}{Irr}
\dmo{\OIR}{OIR}
\dmo{\IOR}{IOR}
\title{Stiefel-Whitney classes for Symmetric Groups}
\author{Sujeet Bhalerao, Jyotirmoy Ganguly, Steven Spallone}

\address{Department of Mathematics, University of Illinois at Urbana-Champaign, Urbana, IL, USA} \email{sujeetbhalerao@gmail.com}

\address{GITAM University,  Bengaluru, NH 207, Nagadenahalli, Doddaballapura Taluk, Bengaluru Rural, Karnataka - 562163, India}
 \email{jyotirmoy.math@gmail.com}  

\address{Steven Spallone, Indian Institute of Science Education and Research, Pune-411008, Maharashtra, India}
\email{sspallone@gmail.com}

\makeatletter\@addtoreset{chapter}{part}\makeatother%

\begin{document}
 \vspace*{-2cm}	
\maketitle

  \begin{abstract} We prove several results about Stiefel-Whitney Classes (SWCs) $w_k(\pi)$ of representations $\pi$ of $S_n$. First, each SWC  is polynomial in the character values of $\pi$ at involutions. Next, for a fixed   $k$, the proportion of irreducible $\pi$ for which $w_k(\pi)=0$  approaches $100\%$ as $n \xrightarrow[]{} \infty$.
A similar result holds for the top SWCs.
We also provide a simple criterion which determines the first nonvanishing SWC for a representation.
The first four SWCs are computed explicitly. Finally, we give analogues for alternating groups.
   \end{abstract}

\tableofcontents

\section{Introduction}
Characteristic classes are the fundamental invariants of vector bundles. The most important characteristic classes for real vector bundles $\mc V$ are the Stiefel-Whitney Classes (SWCs) $w_k(\mc V)$, for nonnegative integers $k$.  To a real representation $\pi$ of a group $G$, one may associate a vector bundle $\mc V_\pi$ over a classifying space of $G$. As in \cite{Benson}, one defines $w_k(\pi)=w_k(\mc V_\pi)$ and calls these the SWCs of $\pi$. 

Initial work on the second SWC for representations of $S_n$ was carried out in \cite{GS}. Many formulas for SWCs for certain finite groups of Lie type can be found in (\cite{MS23}, \cite{MS25}, \cite{MS26},  \cite{GJ23}).  See also \cite{anwesh} for some asymptotics of SWCs for $\GL(n,q)$.  Chern classes and SWCs for connected reductive Lie groups are studied in \cite{Joshi.Lie}. See  \cite{BJM25} for SWCs of dihedral groups.

This paper focuses primarily on SWCs for representations $\pi$ of the symmetric group $S_n$. We prove that each $w_k(\pi)$ can be expressed as a polynomial in the character values of $\pi$ at involutions. For example,  $w_1(\pi) \in \HH^1(S_n) \cong \HHom(S_n,\Z/2\Z)$ gives the determinant of $\pi$, and so $w_1(\pi)=\frac{\chi(1)-\chi(\tau)}{2} v$, where $\tau$ is a transposition and $v$ is the nonzero member of $\HH^1(S_n)$. 
Our main tool is the reduction of the computation of $w_k$ from $S_n$ to elementary abelian $2$-subgroups (EA2Gs) of $S_{2k}$. This is done by combining cohomology detection results of B. M. Mann with the work of Nakaoka. (See Theorems \ref{Es.detect}  and \ref{thm: nakaoka-stability}.)

When these character values are divisible by sufficiently high powers of $2$, the SWCs will vanish since they take values in mod $2$ cohomology. But \cite{GPS} shows that irreducible characters are generically divisible by any positive integer. This leads to our first theorem.
Let $\Lambda_n$ be the partitions of $n$, and put $p(n)=|\Lambda_n|$. Recall that the irreducible representations of $S_n$ are parametrized as $\pi_\la$ for $\la \in \Lambda_n$.
\begin{theorem} \label{thm1}
For a positive integer $k$, we have
\[
        \lim_{n\to\infty}
        \frac{
        \#\{\lambda \in \Lambda_n : w_k(\pi_\lambda)=0\}
        }{p(n)}
        =1.
\]
\end{theorem}
We also apply work of \cite{law2021plethysms} to prove the same (Proposition \ref{top.swc.usually}) about the vanishing of top SWCs, meaning $w_k(\pi)$ when $k=\deg \pi$.   
This result and Theorem \ref{thm1} have analogues for the alternating groups $A_n$; we work this out in Section \ref{section: alternating}.

In addition to asymptotics, we also compute the first nonvanishing SWC. It is easy to see (Proposition \ref{motivational}) that $w(\pi)$ is trivial iff $\pi$ is trivial iff $\chi_\pi(\tau)=\deg \pi$, for a transposition $\tau$. Define $\ord(\pi)$ to be the maximum $k$ so that  $2^{k+1}$ divides $\deg \pi-\chi(\iota)$ for all involutions $\iota \in S_n$ with at most $2^k$ transpositions.   

\begin{thm} \label{thm2} For a nontrivial representation $\pi$ of $S_n$, and $f=\ord(\pi)$, we have $w_i(\pi)=0$ for $1 \leq i<2^f$, and  $w_{2^{f}}(\pi) \neq 0$.
\end{thm}

\bigskip

This is proved in Section \ref{sec: VOSG}. We also compute $w_{2^f}(\pi)$ explicitly; see Scholium \ref{first.nz.comp}. 
From these, we can in principle find character formulas for many SWCs. The character formula for $w_2(\pi)$ from \cite{GS} is rederived in Section \ref{w1and2}. In Section \ref{sec:w3} we give a character formula for $w_3(\pi)$, and in Section \ref{hereisw4} we compute $w_4(\pi)$.

\bigskip
   
\section{Notation and Preliminaries}   
  Let $G$ be a finite group. Until Section \ref{section: alternating},  all representations are understood to be on finite-dimensional real vector spaces.
If $\pi: G \to \GL_n(\R)$ is a representation, write $\det(\pi)$ for the composition $\det \circ \pi$. When $\det(\pi)$ is trivial, we say $\pi$ is \emph{achiral}, otherwise we say $\pi$ is \emph{chiral}.  Let $C_2$ be the  group of order $2$. We write EA2G for ``elementary abelian $2$-group''.
   
   Write $\Pi(G)$ for the set of isomorphism classes of representations of $G$. Given $\pi \in \Pi(G)$, we write $\chi_\pi$ for the character of $\pi$.    
   Let $\mb 1$ be the degree $1$ trivial representation of $G$.
   
   Let $G_1,G_2$ be groups and $(\pi_1,V_1)$ and $(\pi_2,V_2)$ representations of $G_1$ and $G_2$, respectively. Their \emph{external tensor product} $\Pi=\pi_1 \boxtimes \pi_2$ is the representation of $G_1 \times G_2$ on $V_1 \otimes V_2$ given by $\Pi(g_1,g_2)(v_1 \otimes v_2)=\pi_1(g_1)v_1 \otimes \pi_2(g_2)v_2$.

   Write $\HH^k(G)$ for the degree $k$ group cohomology of $G$ with coefficients in $\Z/2\Z$; it is a mod $2$ vector space.
   Recall (e.g., from \cite{Benson}) that for $\pi \in \Pi(G)$ one defines Stiefel-Whitney Classes (SWCs) $w_k(\pi)=w_k^G(\pi) \in \HH^k(G)$ for nonnegative integers $k$.  Note that $w_k(\pi)=0$ when $\deg \pi<k$. Write $w_{\topp}(\pi)$ for $w_k$ when $k=\deg \pi$; this is called the top SWC. When $\pi$ is achiral, it is the reduction mod $2$ of the Euler class of $\pi$.
   When $H$ is a subgroup of $G$, write $\pi|_H$ for the restriction of $\pi$ to $H$, and $w_k^H(\pi) \in \HH^k(H)$ for the SWC of $\pi|_H$.

 Let $S_n$ be the usual group of permutations of $\{1,2, \ldots, n\}$.  For each $0 \leq i \leq \lfloor n/2 \rfloor$, denote by $\iota_i$ the involution $(12) \cdots (2i-1\: \: 2i) \in S_n$.
   In particular, $\iota_0$ is the identity. Let $\pi_{\st}$ be the standard representation of $S_n$, i.e., by permutation matrices.

\section{Character Formulas}\label{CF}
 
\subsection{Terminology}

\begin{defn}\label{def1}
Let $g_1,\ldots, g_n$ be representatives for the conjugacy classes in $G$. We say $f:\Pi(G)\to\Z$ is a \emph{character formula}, when there exists a polynomial $p\in \Q[x_1,\ldots, x_n]$, whose constant term is $0$, so that $f(\pi)= p(\chi_{\pi}(g_1),\ldots,\chi_{\pi}(g_n))$ for all $\pi \in \Pi(G)$.
\end{defn}

\begin{example} \label{mult.char}
Fix an irreducible representation $\sigma$ of $G$.  For $\pi \in \Pi(G)$, write $m_\sigma(\pi)$ for the multiplicity of $\sigma$ in $\pi$. Since
\beq
m_\sigma(\pi)=\frac{1}{|G|}\sum_{g\in G}\chi_{\pi}(g){\ol{\chi_{\sigma}(g)}},
\eeq
 the function $m_\sigma$ has a character formula.
 \end{example}
 
\begin{defn}
Given $k \geq 1$ we say there is a \emph{character formula} for $w_k^G$, when for each linear functional $\vartheta: \HH^k(G) \to \Z/2\Z$, there is a character formula $f_{\vartheta}: \Pi(G) \to \Z$ so that for all representations $\pi$, we have
$\langle \vartheta, w_k(\pi) \rangle= f_{\vartheta}(\pi)\pmod 2$. 
\end{defn}

This amounts to expressing $w_k(\pi)$, in terms of a basis of $\HH^k(G)$, with coefficients which are (integer-valued) polynomials in the character values of $\pi$. Typically these polynomials  involve various multiplicities $m_\sigma(\pi)$.

\begin{example} Let $G=C_2$, and $\pi$ a representation of $G$. Let $v \in \HH^1(C_2)$ be the nonzero element. Then $w(\pi)=(1+v)^b$, where $b$ is the multiplicity of the nontrivial linear character of $C_2$ in $\pi$. Since $w_k(\pi)=\binom{b}{k}v^k$, we see $w_k$ has a character formula.
\end{example}

\subsection{Character Formulas for SWCs: EA2Gs}

\begin{lemma}\label{lem1} Let $E$ be an EA2G.
For each $k \geq 1$, there is a character formula for $w_k^E$.
\end{lemma}
\begin{proof}

Let $n$ be the rank of $E$. Picking a basis gives a factorization $E \cong C_2^n$. Let $\mathcal{N}$ denote the set of subsets of $[n]=\{1,2,\ldots, n\}$.	
    Each representation $\varphi$ of $E$ has the form
    \begin{equation*}
\varphi=\bigoplus_{I\in \mathcal{N}}m_I\sigma_I,
    \end{equation*}	
    where $\sigma_{I}$ is the external tensor product representation $\rho_1\boxtimes\rho_2\boxtimes\cdots\boxtimes \rho_n$  of $E$ with factors
    \beq
    \rho_j=
    \begin{cases}
    \sgn,\quad \text{if $j\in I$},\\
    \mathbb{1},\quad \text{otherwise}.
    \end{cases}
    \eeq
    This gives 
    \begin{equation}\label{tsw}
    w(\varphi)=\prod_{I\in \mathcal{N}}(1+v_I)^{m_I},
    \end{equation}
    where $v_I=w_1(\sigma_I)$. 
By Example \ref{mult.char}, each $m_I$ has a character formula. Expanding  \eqref{tsw}, one obtains
$$w(\varphi)=\prod_{I\in \mathcal{N}}\left( \sum_{j=0}^{m_I} \binom{m_I}{j}v_I^j\right).$$
 This shows that  $w_k$, which is the  $k$-degree homogeneous part of this expression, has a character formula. 
\end{proof}

It will be convenient to rewrite this as follows:

\begin{cor} \label{conv.cor} For every $\vt \in \HH^k(E)^\vee$, there is a polynomial $p_\vt$ in $\ell=|E|$ variables so that   
$$\langle \vt_E, w_k(\pi) \rangle \equiv p_{\vt}(\chi_\pi(e_{1}),..., \chi_{\pi}(e_\ell))\pmod 2,$$
where $E=\{e_1,\ldots, e_\ell\}$. 
\end{cor}

\subsection{Proof of Theorem 1}

We will refer to the maximal elementary abelian $2$-subgroups of $S_n$ as simply the ``EA2Gs of $S_n$''; they are described in more detail in Section \ref{EA2Gs.sn}.
 
Let us recall two major ``detection'' results.
 
\begin{thm} \cite{mann}, \cite[Theorem 1.2, page 179]{ademmil} \label{Es.detect} The restriction map
\begin{equation}\label{eq3}
\HH^k(S_n)\to \bigoplus_{E} \HH^k(E),
\end{equation}
where $E$ runs over the EA2Gs of $S_n$, is injective.
  \end{thm}

\begin{thm}\label{thm: nakaoka-stability} \cite[Corollary 6.7]{naka}
For $n \geq 2k$, the restriction map $\HH^k(S_n) \to \HH^k(S_{2k})$
is an isomorphism.
\end{thm}

\begin{prop}\label{chf1}
 
    Let $k,n$ be positive integers, and put $m=\left \lfloor n/2 \right \rfloor$. For each  $\vartheta \in \HH^k(S_n)^\vee$, there is an integer-valued polynomial $q_\vt$ in $m+1$ variables, with no constant term, so that for each $\pi \in \Pi(S_n)$ we have
    \beq
    q_\vt(\chi_\pi(\iota_0),\chi_\pi(\iota_1),..., \chi_{\pi}(\iota_m)) \equiv \langle \vt, w_k(\pi) \rangle \mod 2.
    \eeq

\end{prop}

\begin{proof}
	
 Fix $\vartheta \in \HH^k(S_n)^\vee$. Since  \eqref{eq3} is an injection, $\vartheta$ can be extended to a functional $\widetilde{\vartheta}$ on $\bigoplus\limits_{E }\HH^k(E)$. Then
 \begin{equation}\label{eq4}
 \widetilde{\vartheta}=\Sigma_{E}\vartheta_{E},
 \end{equation}
 where $E$ varies over the EA2Gs of $S_n$, and $\vartheta_{E} \in \HH^k(E)^\vee$. (More precisely $\vt_E$ is projection to $\HH^k(E)$ followed by a linear functional.) By Corollary \ref{conv.cor}, there is a polynomial $p_E$ in $\ell=|E|$ variables so that we  
$$\langle w_k(\pi),\vartheta_{E}\rangle \equiv p_{E}(\chi_\pi(e_{1}),..., \chi_{\pi}(e_\ell))\pmod 2,$$
where $E=\{e_1,\ldots, e_\ell\}$. Grouping elements of $E$ which are conjugate in $S_n$ gives a polynomial $q_E \in \Q[x_0,x_1, \ldots, x_m]$ with
\beq
q_{E}(\chi_\pi(\iota_0),\chi_\pi(\iota_{1}),..., \chi_{\pi}(\iota_m))=p_{E}(\chi_\pi(e_{1}),..., \chi_{\pi}(e_\ell)).
\eeq
Using Equation \eqref{eq4} we obtain
\beq
\begin{split}
\langle w_k(\pi),\vartheta\rangle &=  \langle w_k(\pi),\widetilde{\vartheta}\rangle  \\
  &=\langle w_k(\pi),\Sigma_{E}\vartheta_{E}\rangle \\
  &\equiv\sum_{E} q_E(\chi_\pi(\iota_0),..., \chi_{\pi}(\iota_m))\pmod 2 \\
  &\equiv q_\vt(\chi_\pi(\iota_0),..., \chi_{\pi}(\iota_m)), \\
  \end{split}
  \eeq
   where $q_\vt=\sum_E q_E$.
\end{proof}

 \begin{Scholium} We have  shown that for each $k \geq 1$, there is a character formula for $w_k^{S_n}$.
 \end{Scholium}

From Proposition \ref{chf1} we deduce our first theorem.

\begin{proof} (of Theorem \ref{thm1}) Let $n \geq 2k$. By Theorem \ref{thm: nakaoka-stability}, $w_k(\pi_\la)=0$ iff for all $\vt \in \HH^k(S_{2k})^\vee$ we have $\lip \vt,w_k(\pi_\la|_{S_{2k}} )\rip=0$.
By Proposition \ref{chf1}, there are integer-valued polynomials $q_\vt$ in $k+1$ variables, with no constant term, so that for all $\pi \in \Pi(S_{2k})$
\beq
q_\vt(\chi_\pi(\iota_0),..., \chi_{\pi}(\iota_k)) \equiv \lip \vt,w_k(\pi) \rip \mod 2.
\eeq
By clearing denominators, there are positive integers $d_\vt$ so that each $d_\vt q_\vt(x_0, x_1, \ldots, x_k) \in \Z[x_0, x_1, \ldots, x_k]$.
Let $D$ be the product of the $d_\vt$ for $\vt \in \HH^k(S_{2k})^\vee$.
From  \cite{GPS} we know
\begin{equation} \label{frum.gps}
\lim_{n\rightarrow \infty}\dfrac{\#\{\lambda\vdash n\mid \chi_{\lambda}(\iota_i) \text{ is divisible by } 2 D  \:\:\: \forall   0 \leq i \leq k\}}{p(n)}=1.
\end{equation}
 
Now suppose $\pi \in \Pi(S_{n})$, and each $\chi_\pi(\iota_i)$ is divisible by $2D$. Then $q_\vt(\chi_\pi(\iota_0),..., \chi_{\pi}(\iota_k))$ is even, which entails that
$\lip \vt,w_k(\pi|_{S_{2k}}) \rip=0$. This being true for all $\vt$, it must be that $w_k(\pi)=0$.  
Hence
\beq
\dfrac{\#\{\lambda\vdash n\mid \chi_{\lambda}(\iota_i) \text{ is divisible by } 2D \:\:\: \forall   0 \leq i \leq k \}}{p(n)} \leq   \frac{
        \#\{\lambda\vdash n : w_k(\pi_\lambda)=0\}
        }{p(n)},
        \eeq
        so the conclusion follows from \eqref{frum.gps}.
\end{proof}

   \section{SWCs for Elementary Abelian $2$-groups}
   
 \subsection{Cohomology}
  
 The mod $2$ cohomology of $C_2$ is polynomial in $v=w_1(\sgn) \in \HH^1(C_2)$, so we simply write 
  $\HH^*(C_2)=(\Z/2\Z)[v]$. Let $E$ be an elementary abelian $2$-group with basis $e_1, \ldots, e_r$, and let $e_1^*, \ldots, e_r^*$ be the dual basis. 
  If we put $v_i=\sgn(e_i^*)$, then by K\"{u}nneth, the mod $2$ cohomology of $E$ is polynomial in the $v_i$, in other words
  $\HH^*(E)=(\Z/2\Z)[v_1, \ldots, v_r]$.  
  In fact, the first SWC can be regarded as an isomorphism $w_1: E^\vee \overset{\sim}{\to} \HH^1(E)$, which extends to an identification $\Sym^* E^\vee  \overset{\sim}{\to} \HH^*(E)$.
  
  In what follows, we sometimes write $|\alpha|=k$ when $\alpha \in \HH^k(E)$.
 
 \subsection{Vanishing Order of a Representation}

\begin{defn} \label{order.ea2g} Given a representation $\pi$ of $E$, put
\beq
\ord(\pi)=\max\{ k \geq 0 \mid \forall e \in E, \: \:   \chi_\pi(e) \equiv \deg \pi \mod 2^{k+1}\}.
\eeq
Put $\ord(\pi)=\infty$ when $\pi$ is trivial.
 \end{defn}
  
 \begin{example} If $\vartheta \in E^\vee$ is nontrivial, then $\ord(\vartheta)=0$; moreover 
 $\ord \left( \vartheta^{\oplus 2^a} \right)=a$ for any $a \geq 0$.
 \end{example}

 \begin{prop} \label{c2.vanishing.order} Let $E=C_2$. 
 \begin{enumerate}
\item If $\pi$ is nontrivial, then $w_1(\pi)= \cdots =w_{2^k}(\pi)=0$ iff $\ord(\pi) \geq k+1$.
\item If $w(\pi)=1$, then $\pi$ is trivial. 
\end{enumerate}
 \end{prop}
   
 \begin{proof}  Let $0 \neq \vartheta \in E^\vee$, and $v=w_1(\vartheta)$. Then $w(\pi)=(1+v)^m$, where $m=\frac{\deg \pi-\chi(g)}{2}$.
    The binomial coefficients $\binom{m}{i}$ are even for all $0<i \leq 2^k$ iff $m$ is a multiple of $2^{k+1}$. The second statement follows from the first.
    \end{proof}

  For $E$ arbitrary, by restricting to cyclic subgroups we have:
   
 \begin{cor} \label{binom.mod.2}  
 If $w_1(\pi)=\cdots =w_{2^k}(\pi)=0$, then $\ord(\pi) \geq k+1$. 
     \end{cor}
  
  \bigskip
  We now start building to the converse of this. (See Section \ref{next.secshun}.)
   
 For a given $e \in E$, consider the virtual representation
   \beq
    \pi_e=\sum_{v \in E^\vee} {\lip v,e \rip}v,
    \eeq
meaning that the linear representation $v$ occurs in $\pi_e$ with multiplicity $\pm 1={\lip v,e \rip}$.
In particular, $\pi_0$ is the regular representation of $E$.

\begin{prop} Write $\chi_e$ for the character of $\pi_e$.
\begin{enumerate}
\item If $e \neq e'$, then $\chi_e(e')=0$.
\item $\chi_e(e)=|E|$.
\end{enumerate}
\end{prop}

\begin{proof}
This is because
\beq
\chi_e(e')=\sum_{v \in E^\vee} \lip v,e \rip \lip v,e' \rip=\sum_{v \in E^\vee} \lip v,e-e' \rip,
\eeq
which vanishes iff $e=e'$.
\end{proof}

Consequentially:
\begin{prop} \label{30s} Let $\pi$ be a representation of $E$ with $ \rank E - 1 \leq \ord(\pi)$, and for each $e \in E$, put $m_e=\frac{\chi(e)-\deg \pi}{|E|}$. Then as virtual representations we have
\beq
\pi- (\deg \pi) {\mb 1}=\sum_e m_e \pi_e.
\eeq
\end{prop}

\subsection{Dickson Invariants}
The \emph{Dickson product} for $E$ is defined as
\beq
\mc D(E)= \prod_{v \in E^\vee} (1+v) \in \Sym^*(E^\vee).
\eeq
It may be identified with the total SWC of the regular representation of $E$.

The nonzero homogeneous components of $\mc D(E)$ are  $\GL(E)$-invariant polynomials called \emph{Dickson invariants}. Let $r=\rank E$; the invariants occur in degrees $2^{r-1}=2^r-2^{r-1}, 2^r-2^{r-2}, \ldots, 2^r-1$. Let us write $d_i(E)$ for the Dickson invariant of degree $2^r-2^{r-i}$, for $1 \leq i \leq r$. For example, $d_r(E)$ is the product of the nonzero members of  $E^\vee$.

 \begin{example}
 For $r=2$, let $v_1,v_2$ be a basis of $E^\vee$; then $d_1(E)=v_1^2+v_1v_2+v_2^2$  and $d_2(E)=v_1^2v_2+v_1v_2^2$.
 For $r=3$, let  $v_1,v_2,v_3$ be a basis of $E^\vee$; then
\begin{equation} \label{rk3.dickson}
d_1(E)=v_1^4+v_1^2(v_2^2+v_2v_3+v_3^2)+v_1(v_2^2v_3+v_2 v_3^2)+v_2^4+v_2^2v_3^2+v_3^4.
\end{equation}
\end{example}

For more details and examples, see \cite{wilkerson} and \cite{ademmil}. Our $d_i(E)$ is written as $c_{r,r-i}$ in \cite{wilkerson}. Most of our computations only involve the first Dickson invariant, so write $d_E=d_1(E)$; its degree is $2^{r-1}$.
 
\subsection{Vanishing Orders and SWCs} \label{next.secshun}
 \begin{defn}   For $e \neq 0$ in $E$, write $e^\perp \subset E^\vee$ for the hyperplane of functionals vanishing on  $e$. Write $d_e=d_1(e^\perp) \in \Sym^* e^\perp \subset \Sym^* E^\vee$.
 \end{defn}
 Note that $\deg d_e=2^{r-2}$. We have
  \beq
  \begin{split}
  w(\pi_e) &=\frac{\Pi_{v \in e^\perp}(1+v)}{\Pi_{v \notin e^\perp}(1+v)} \\
  &= \frac{ \mc D(e^\perp)^2}{\mc D(E^\vee)} \\
        &= 1+ d_E + d_{e}^2+ HOT. \\
        \end{split}
  \eeq
 (Here `$HOT$' means ``higher order terms''.)  Define $d_0 = 0$. Write $\mc D_e=d_E+d_e^2$; we have shown $w_{2^{r-1}}(\pi_e)=\mc D_e$.
  
  \begin{prop} \label{sum.de.vanish} We have
  $\sum_e d_e=0$.
  \end{prop}
  
  \begin{proof} This is because the sum would be $\GL(E)$-invariant, but there are not such invariants of degree $2^{r-2}$.
  \end{proof}

 \begin{thm} \label{summer}
 Let $E$ be an EA2G of rank $r$, and $\pi$ a representation of $E$. Suppose  $f=\ord(\pi)  \geq r-1$. Then $w_i(\pi)=0$ for $1 \leq i<2^f$, and
   \begin{equation} \label{thiss}
   w_{2^{f}}(\pi)=\sum_e \frac{\deg \pi -\chi_\pi(e)}{2^{f+1}} \mc D_e^{2^{f-r+1}}.
   \end{equation}
 \end{thm}
 \begin{proof}
Each $m_e=\frac{\chi(e)-\deg \pi}{2^r}$  is divisible by $2^{f-r+1}$, so define 
 \beq
 m_e'=m_e/2^{f-r+1}=\frac{\chi_\pi(e)-\deg \pi}{2^{f+1}}.
 \eeq
 
 By Proposition \ref{30s}, we have 
 \beq
 \begin{split}
 w(\pi) &=\prod_e (1+ \mc D_e + HOT)^{m_e} \\
        &= \prod_e (1+ \mc D_e^{2^{f-r+1}} +HOT)^{m_e'} \\
        \end{split}
        \eeq
		
The conclusion follows since $|\mc D_e^{2^{f-r+1}}|=2^{f}$. \end{proof}

\begin{example} Let $\ord(\pi) \geq 1$.
For $\rank E=2$, we have
\beq
w_2(\pi)=\sum_e  \frac{\chi_\pi(e)-\deg \pi}{4} \mc D_e.
\eeq
\end{example}

  \begin{cor} \label{schmabel} Let $f=\ord(\pi)$. Then $w_i(\pi)=0$ for $1 \leq i<2^f$, and if $r \geq f+2$, then
\beq
w_{2^f}(\pi)^{2^{r-f-2}}=\sum_e \frac{\chi_\pi(e)-\deg \pi}{2^{f+1}}d_e.
\eeq
  
\end{cor}

\begin{proof}
  If $r \leq f+1$ we are done by Theorem \ref{summer}. Else $r \geq f+2$. Put $\Pi=\pi^{\oplus 2^{r-f-1}}$, so that
$\ord(\Pi)=r-1$. By Theorem \ref{summer}, $w_i(\Pi)=0$ for $1 \leq i<2^{r-1}$, and
\beq
\begin{split}
w_{2^{r-1}}(\Pi) &=\sum_e \frac{\chi_\Pi(e)-\deg \Pi}{2^{r}} \mc D_e \\
&= \sum_e \frac{\chi_\pi(e)-\deg \pi}{2^{f+1}}\mc D_e. \\
\end{split}
\eeq

Now for $1 \leq k$ we have 
\begin{equation}  
w_{k2^{r-f-1}}(\Pi)=w_k(\pi)^{2^{r-f-1}}.
\end{equation}
This entails that $w_i(\pi)=0$ for $1 \leq i<2^f$, and
\beq
w_{2^f}(\pi)^{2^{r-f-1}} =\sum_e \frac{\chi_\pi(e)-\deg \pi}{2^{f+1}}\mc D_e. 
\eeq
  
Now $\sum_e \chi_\pi(e)= |E|\mu_0$,
 where $\mu_0$ is the multiplicity of the trivial representation in $\pi$.
 Therefore
 \beq
 \sum_e \frac{\chi_\pi(e)-\deg \pi}{2^{f+1}}=\frac{|E|}{2^{f+1}}(\mu_0 - \deg \pi),
 \eeq
 which vanishes mod $2$.
Therefore
\beq
w_{2^f}(\pi)^{2^{r-f-1}} =\sum_e \frac{\chi_\pi(e)-\deg \pi}{2^{f+1}}d_e^2, 
\eeq
  so by taking square roots, we arrive at
\beq
w_{2^f}(\pi)^{2^{r-f-2}}=\sum_e \frac{\chi_\pi(e)-\deg \pi}{2^{f+1}}d_e.
\eeq

\end{proof}

\subsection{Other SWCs}

In this section we gather tools which permit the calculation of the SWCs $w_k(\pi)$, when $k$ is not a power of $2$. The main one is Wu's Formula:
\begin{prop} \cite[page 94]{mch} 
 For a representation $\pi$ of a group $G$, and $i \geq 1$ we have
 \begin{equation} \label{wenjun}
 \Sq^i(w_m(\pi))=w_i(\pi) w_m(\pi) + \binom{i-m}{1}w_{i-1}(\pi)w_{m+1}(\pi)+ \cdots + \binom{i-m}{i} w_0(\pi)w_{m+i}(\pi).
 \end{equation}
  \end{prop}
 
 Often $i<m$, so we must understand the notation $\binom{x}{k}=\frac{x(x-1)\cdots(x-k+1)}{k!}$ when $x$ is negative.
 For example, $w_3(\pi)=\Sq^1(w_2(\pi))+w_1(\pi) \cup w_2(\pi)$. 
 
 Similarly one may use \eqref{wenjun} to express a given $w_j(\pi)$ in terms of $\Sq^{t}(w_{2^k}(\pi))$, with $t<2^k< j$, and SWCs $w_i(\pi)$ with $i<j$.
 Steenrod squares on the Dickson invariants have been computed. Below we write $d_i$ for $d_i(E)$.
\begin{prop}  \label{wilk} \cite[Corollary 2.4]{wilkerson} Let $k \geq 0$. Then
\begin{enumerate}
\item $\Sq^{2^k} d_i=0$ unless $i=r-1-k$ or $r = k+1$
\item $\Sq^{2^k} d_{r-1-k}=d_{r-k}$.
\item $\Sq^{2^k} d_i=d_1d_i$ if $r=k+1$.
\end{enumerate}
\end{prop} 

 Other Steenrod powers can be computed by combining the above with  the Adem-Wu relations. In view of the exponent appearing in Theorem \ref{summer}, we also note the following, which follows from the Cartan formula.

\begin{lemma}  \label{blum}
If $b \in \HH^*(G)$ and $\ord(i)<\ord(j)$, then
\beq
\Sq^i(b^j)=0.
\eeq
We have, for all $u$,
\beq
\Sq^{2^ku}(b^{2^k})=\Sq^u(b)^{2^k}.
\eeq
\end{lemma}

 In the situation of Theorem \ref{summer}, it is enough to compute Steenrod squares on $\mc D_e^{2^{f-r+1}}$.
By Lemma \ref{blum}, if $\Sq^i(\mc D_e^{2^{f-r+1}})$ is nonzero, then $f-r+1 \leq \ord_2(i) \leq f$,
and in that case the computation reduces to Steenrod squares on $d_E$ and $d_e$.
 
 One can similarly treat the situation of Corollary \ref{schmabel}; we omit the details.
 
\begin{prop} \label{machiav} Suppose $E$ has rank $2$ and $\pi$ is achiral. Then
  \begin{equation} \label{w3.ach}
  w_3(\pi)=\left(\sum_{e \in E} \frac{\deg \pi- \chi_\pi(e)}{4} \right)d_2(E).
  \end{equation}
  \end{prop}
 
 \begin{proof} 
 We must compute  $w_3(\pi)=\Sq^1(w_2(\pi))$, where
\beq
w_2(\pi)=\sum_{e \in E} \frac{\deg \pi- \chi_\pi(e)}{4} (d_E+d_e^2),
\eeq
By the Cartan formula (or Lemma \ref{blum}), $\Sq^1(d_e^2)=0$, and by Proposition \ref{wilk}, 
  $\Sq^1(d_E)=d_2(E)$.  
 \end{proof}

 \section{Vanishing Orders for the Symmetric Group} \label{sec: VOSG}
 \subsection{Vanishing Order}
We begin with an easy but motivational proposition.
For $n \geq 2$, recall that $\iota_1$ is the transposition $(12) \in S_n$. 
 \begin{prop} \label{motivational} For a representation $\pi$ of $S_n$,
 the following are equivalent:
 \begin{enumerate}
 \item $w(\pi)=1$.
 \item The restriction of $\pi$ to $S_2$ is trivial.
  \item $\chi_\pi(\iota_1)=\deg \pi$.
   \item $\pi$ is trivial.
 \end{enumerate}
 \end{prop}
 
 \begin{proof} 
 If $w(\pi)=1$, then also $w^{S_2}(\pi)=1$. By Proposition \ref{c2.vanishing.order}, the restriction of $\pi$ to $S_2$ is trivial, which is equivalent to $\chi_\pi(\iota_1)=\deg \pi$. In this case, $\iota_1 \in \ker \pi$ and therefore $\pi$ is trivial. 
 \end{proof}

 \begin{defn} Let $\pi$ be a nontrivial representation of $S_n$. 
We define
 \beq
 \ord(\pi)=\max \left\{ k \geq 0 :   \forall 1 \leq i \leq   \min \left(2^k, \left\lfloor \frac{n}{2} \right \rfloor \right), \chi_\pi(\iota_i) \equiv \deg \pi \mod 2^{k+1} \right\}.
 \eeq
 When $\pi$ is trivial, put $\ord(\pi)=\infty$.
\end{defn}
 
 In other words, it is the maximum $k$ so that $2^{k+1}$ divides $\deg \pi-\chi(\iota)$ for all involutions $\iota \in S_n$ with at most $2^k$ transpositions.
 
 \begin{example} For the standard representation $\pi_{\st}$, we have $\ord(\pi_{\st})=0$. For an integer $a \geq 0$, we have $\ord\left(\pi_{\st}^{\oplus 2^a} \right)=a$.
 \end{example}

  Now we deduce our second main theorem.
  
\begin{proof} (of Theorem \ref{thm2})
Let $f=\ord(\pi)$ and $1 \leq i<2^f$; we must show that $w_i(\pi)=0$. 
Suppose first that $n \geq 2i$. By Theorem \ref{thm: nakaoka-stability}, it is enough to prove that 
$w^{S_{2i}}_i(\pi)=0$.  If $E<S_{2i}$ is an EA2G, then every $e \in E$ is conjugate to some $\iota_j$ with $0 \leq j \leq  i$. 
 By the definition of $\ord(\pi)$, we have $\chi_\pi(e) \equiv \deg \pi \mod 2^{f+1}$ for all $e \in E$, hence $\ord(\pi|_E) \geq f$.
Therefore $w_{i}^E(\pi)=0$  by  Corollary \ref{schmabel}, and then $w_i^{S_{2i}}(\pi)=0$ by Theorem \ref{Es.detect}. 
\newline
If instead $n<2i$, then any involution in $S_n$ is conjugate to an $\iota_j$ with $0 \leq j<i$, so the same argument goes through.

Suppose also $w_{2^{f}}(\pi)=0$.   
By Corollary \ref{binom.mod.2}, $\chi_\pi(\iota) \equiv \deg(\pi) \mod 2^{f+2}$ for any involution $\iota \in S_n$. 
 Hence $\ord(\pi)>f$, a contradiction. \end{proof}

    \subsection{Elementary Abelian $2$-groups of $S_n$} \label{EA2Gs.sn}
     
The action of $C_2^k$ on itself by translation gives an injection $C_2^k \hookrightarrow \Perm(C_2^k)$. We identify  $\Perm(C_2^k)$ with $S_{2^k}$, and write $E_k<S_{2^k}$ for the image of this injection. It is well-defined up to conjugacy. Then $E_k$ is a maximal elementary abelian $2$-subgroup, and  every nonidentity element of $E_k$ is conjugate to $\iota_{2^{k-1}}$.
 
 For a tuple of positive integers ${\bm d}=(d_1, \ldots, d_r)$ be an $r$-tuple of positive integers, write $2^{\bm d}=2^{d_1}+ \cdots + 2^{d_r}$. 
 When $2^{\bm d}=n$, we have inside the Young subgroup
 \beq
 S_{2^{d_1}} \times \cdots \times S_{2^{d_r}} < S_n, 
 \eeq
 the subgroup
 \beq
 E_{\bm d}=E_{d_1} \times \cdots \times E_{d_r} < S_n.
 \eeq
  Note that the rank of $E_{\bm d}$ is $\sum d_i$. When $n$ is odd, we similarly have subgroups $E_{\bm d} <S_{n-1}<S_n$ when $2^{\bm d}=n-1$.
 
 \begin{prop} \cite[page 179]{ademmil} Let $n$ be a positive integer. Every maximal elementary $2$-subgroup of $S_n$ is conjugate to some $E_{\bm d}$, with $2^{\bm d}=n$ when $n$ is even, or $2^{\bm d}=n-1$ when $n$ is odd.
 \end{prop}
 
One always has ${\bm 1}=(1, \ldots, 1)$ with $\lfloor n/2 \rfloor$ parts; then $E_{\bm 1}$ is the subgroup of $S_n$ generated by $\iota_1, \ldots, \iota_{\lfloor n/2 \rfloor}$. 
Note that $E_{\bm d} \leq A_n \Leftrightarrow$  each $d_i>1$. When $\pi$ is a representation of $S_n$, and $2^{\bm d}=n$, write $w^{\bm d}(\pi) \in \HH^*(E_{\bm d})$ for the restriction to $E_{\bm d}$ of the SWC of $\pi$.

Let $n \geq 2k$. The sum of the restriction maps
\beq
\HH^k(S_n) \to \bigoplus_{\bm d: 2^{\bm d}=2k} \HH^k(E_{\bm d})
\eeq
is injective. Hence computing each $w^{\bm d}_k(\pi)$ is equivalent to computing $w_k(\pi)$.
 
 \begin{Scholium}  \label{first.nz.comp} Let $\pi$ be a representation of $S_n$.  For $\bm d$ with $2^{\bm d}\leq n$, put $f_{\bm d}=\ord(\pi|_{E_{\bm d}})$, as in Definition \ref{order.ea2g}. If $f=\ord(\pi)$, and $n \geq 2^{f+1}$, then also $f=\min( f_{\bm d} \mid 2^{\bm d}=2^{f+1})$. When $f_{\bm d}>f$, then $w_{2^f}^{\bm d}(\pi)=0$. When $f_{\bm d}=f$, then $w_{2^f}^{\bm d}(\pi)$ is computed in Proposition \ref{summer} or Corollary \ref{schmabel}.
In this sense, we have computed the first nonvanishing SWC $w_{2^f}(\pi)$.
 \end{Scholium}
  
  Our next goal is to compute explicitly the SWCs up to $w_4$.

\subsection{First and Second SWC} \label{w1and2}

The first SWC is determined by the restriction to $S_2$. Let $v \in \HH^1(S_2)$ be the nontrivial element.
Then
\beq
w_1^{{S_2}}(\pi)=\frac{\deg \pi-\chi_\pi(\iota_1)}{2} v.
\eeq

The second SWC is detected by the restriction to $S_4$. Here the EA2Gs are $E=E_{\bf 1}$ and $E_{(2)}$, 
both with rank $2$.
When $\ord(\pi)=1$, we have
\beq
w_2^E(\pi)=\sum_{e \in E} \frac{\deg \pi -\chi_\pi(e)}{4} \mc D_e.
\eeq

 We have $d_{e_1}=v_2$, $d_{e_2}=v_1$, and $d_{e_1+e_2}=v_1+v_2$. Moreover $d_E=v_1^2+v_1v_2+v_2^2$.
 Hence $\mc D_{e_1}=v_1^2+v_1v_2$, and $\mc D_{e_2}=v_1v_2+v_2^2$, and $\mc D_{e_1+e_2}=v_1v_2$.
 This gives
 \beq
  w_2^{\bf 1}(\pi)= \frac{\deg \pi -\chi_\pi(\iota_1)}{4} (v_1^2+v_2^2)+   \frac{\deg \pi -\chi_\pi(\iota_2)}{4} (v_1v_2).
  \eeq
Since all non-identity elements of $E_{(2)}$ are conjugate in $S_4$ to $\iota_2$, we obtain
 \begin{equation} \label{w2ellip}
 w_2^{(2)}(\pi)= \frac{\deg \pi - \chi_\pi(\iota_2) }{4} d_{E_{(2)}},
 \end{equation}
using Proposition \ref{sum.de.vanish}.

\bigskip

On the other hand, if $\ord(\pi)=0$ then $\det \pi=\sgn$. Since $w_1^{\bf 1}(\sgn)=v_1+v_2$ and $w_1^{(2)}(\sgn)=0$, this gives
 \begin{equation} \label{canpake}
  w_2^{\bf 1}(\pi)=     \frac{\deg \pi - \chi_\pi(\iota_1)-2}{4} (v_1^2+v_2^2)+   \frac{\deg \pi -\chi_\pi(\iota_2)}{4} (v_1v_2),
  \end{equation}
and $w_2^{(2)}(\pi)$ is again given by \eqref{w2ellip}.

Put $\rho_1=\sgn \oplus \sgn$ and $\rho_2 =\pi_{\st} \oplus \sgn$. Note that $w_2(\rho_2)=w_2(\pi_{\st})+w_1(\sgn)^2$.
Since $\ord(\rho_1)=\ord(\rho_2)=1$, and
 $\HH^2(S_4)$ is $2$-dimensional (e.g., \cite[Section VI.1, page 185]{ademmil}), we can say:

\begin{prop} \label{eps-rho} Let $\pi$ be a representation of $S_n$ with $\ord(\pi) \geq 1$. There exist unique  $a_1,a_2 \in \{0,1\}$ so that if we set 
\beq
\pi^+=\pi \oplus a_1 \rho_1  \oplus a_2 \rho_2,
\eeq
then $\ord(\pi^+) \geq 2$.
\end{prop}
Please note: 
\begin{equation} \label{eps1.formulah}
a_1= \frac{\chi_{\pi}(\iota_2) - \chi_\pi(\iota_1)}{4} \mod 2,
\end{equation}
and
\begin{equation} \label{eps2.formulah}
a_2= \frac{\deg \pi - \chi_\pi(\iota_2)}{4} \mod 2.
\end{equation}
This will be useful in computing the fourth SWCs.

 \bigskip

Equation \eqref{w2ellip} generalizes:
\begin{prop} \label{elliptic.swc.here} For $E=E_{k+1}<S_{2^{k+1}}$, if $\ord(\pi)=k$, then  
\beq
w_{2^k}^{E}(\pi)=\frac{\deg \pi-\chi_\pi(\iota_{2^k})}{2^{k+1}} d_{E}.
\eeq
\end{prop}

\section{Third SWC} \label{sec:w3}

  For $n \geq 6$,  the restriction $\HH^3(S_n)\to \HH^3(S_6)$ is an isomorphism by  \Cref{thm: nakaoka-stability}. Hence it suffices to compute the restriction of $w_3(\pi)$ to $S_6$.
 Wu's formula says that
\begin{equation} \label{wu3}
w_3(\pi)=\Sq^1(w_2(\pi))+ w_1(\pi) \cup w_2(\pi).
\end{equation}

Let $\mc W \subseteq \HH^3(S_6)$ be the sum of two subspaces: the image of $\Sq^1: \HH^2(S_6) \to \HH^3(S_6)$, and the product of $\HH^1(S_6)$ and $\HH^2(S_6)$. By \eqref{wu3}, all third SWCs lie in $\mc W$. 
Write $\Res: \HH^3(S_6) \to \HH^3(S_4)$ for the usual restriction map. To specify $w_3(\pi)$, it is enough to compute $\Res(w_3(\pi))$, because of the following lemma:

\begin{lemma} The map $\Res$ restricts to an injection $\mc W \to \HH^3(S_4)$.
\end{lemma}

\begin{proof}

Let $u = w_1(\sgn) \in \HH^1(S_6)$. Choose the following basis for $\HH^2(S_6)$: $$a = u^2, \quad b= w_2(\rho_2),$$ where $\rho_2 = \pi_{\st} \oplus \sgn$.  
Write
$\HH^*(E_{\bf 1})=\F_2[x,y]$ and $\HH^*(E_{(2)})=\F_2[r,s]$.

The restrictions of $u$ are $u|_{E_{\bf 1}}=x+y$ and $u|_{E_{(2)}}=0$. As for $b$, on
$E_{\bf 1}$ the nontrivial characters in $\rho_2$ are $x,y,x+y$, hence
$$b|_{E_{\bf 1}}=xy+x(x+y)+y(x+y)=x^2+xy+y^2.$$ On $E_{(2)}$, the sign character is
trivial and the nontrivial characters in $\rho_2$ are $r,s,r+s$, hence
$$b|_{E_{(2)}}=rs+r(r+s)+s(r+s)=r^2+rs+s^2.$$

We have
$\Sq^1(a)=\Sq^1(u^2)=0$ and
$$\Sq^1(b)|_{E_{(2)}}=\Sq^1(r^2+rs+s^2)=r^2s+rs^2\neq0,$$
so that 
$$\im(\Sq^1:\HH^2(S_6)\to \HH^3(S_6))=\langle\Sq^1(b)\rangle.$$
Also
$$\HH^1(S_6)\cdot \HH^2(S_6)=u\HH^2(S_6)=\langle ua,ub\rangle=\langle u^3,ub\rangle.$$
Thus
$$\mc W=\langle u^3,ub,\Sq^1(b)\rangle.$$

The restrictions are:
\begin{center}
\begin{tabular}{c|cc}
 & $E_{\bf 1}$ & $E_{(2)}$ \\
\hline
$u^3$ & $(x+y)^3$ & $0$ \\
$ub$ & $(x+y)(x^2+xy+y^2)=x^3+y^3$ & $0$ \\
$\Sq^1b$ & $x^2y+xy^2$ & $r^2s+rs^2$
\end{tabular}
\end{center}

Now suppose $\Res(\alpha u^3+\beta ub+\gamma\Sq^1b)=0$. Restriction to
$E_{(2)}$ gives $\gamma(r^2s+rs^2)=0$, so $\gamma=0$, and restriction to $E_{\bf 1}$
  gives $\alpha(x+y)^3+\beta(x^3+y^3)=0$. Since $(x+y)^3=x^3+x^2y+xy^2+y^3$
and $x^3+y^3$ are independent in $\F_2[x,y]$, we deduce $\alpha=\beta=0$.
Hence $\Res$ is injective on $\mc W$.
\end{proof}

Hence it suffices to compute the restriction of $w_3(\pi)$ to the EA2Gs of $S_4$. Both $E_{{\bf 1}}$ and $E_{(2)}$ have rank $2$, so we will apply Proposition \ref{machiav}.  Since the restriction of $\pi$ to $E_{(2)}$ is achiral, we have
  \beq
   \begin{split}
  w_3^{(2)}(\pi) &=\left(\sum_{e \in E} \frac{\chi_\pi(e)-\deg \pi}{4} \right)d_2(E_{(2)}) \\
  &=  \frac{\chi_\pi(\iota_2)-\deg \pi}{4}d_2(E_{(2)}).  \\
  \end{split}
  \eeq 
 
 Next, let $E=E_{\bf 1}$.
  When $\pi$ is achiral, Equation \eqref{w3.ach} gives
  \beq
  \begin{split}
  w_3^{\bf 1}(\pi) &=\frac{2 \chi(\iota_1)+\chi(\iota_2)-3 \deg \pi}{4} d_2(E_{\bf 1}) \\
   &=\frac{\deg \pi- \chi(\iota_2)}{4}d_2(E_{\bf 1}) .
  \end{split}
  \eeq

When $\pi$ is chiral, to compute $w_3=\Sq^1w_2+w_1 w_2$, we use
  $w_1^{\bf 1}(\pi)=w_1^{\bf 1}(\sgn)=v_1+v_2$ and 
  \beq
    w_2^{\bf 1}(\pi) = \frac{\deg \pi-\chi(\iota_1)-2}{4}(v_1+v_2)^2+ \frac{\deg \pi-\chi(\iota_2)}{4} v_1v_2
    \eeq
    by \eqref{canpake}.
   Since $\Sq^1$ vanishes on squares, this gives $\Sq^1(w_2^{\bf 1}(\pi)) =\frac{\deg \pi-\chi(\iota_2)}{4} d_2(E_{\bf 1})$. 
It follows that
\beq
w_3^{\bf 1}(\pi)= \frac{\deg \pi-\chi(\iota_1)-2}{4}(v_1+v_2)^3.
\eeq

\begin{thm} Let $\pi$ be a representation of $S_n$, with $n \geq 6$.  Then $w_3(\pi) \in  \HH^3(S_n)$ is the unique cohomology class in $\mc W$ which restricts to the following  in $\HH^3(E_{(2)})$ and $\HH^3(E_{\bf 1})$:
\begin{itemize}
\item $\dfrac{\deg \pi- \chi_\pi(\iota_2)}{4}d_2(E_{(2)}) \in \HH^3(E_{(2)})$,
\item $\dfrac{\deg \pi- \chi(\iota_2)}{4}d_2(E_{\bf 1}) \in \HH^3(E_{\bf 1})$ when $\pi$ is achiral,
\item $ \dfrac{\deg \pi-\chi(\iota_1)-2}{4}(v_1+v_2)^3 \in \HH^3(E_{\bf 1})$ when $\pi$ is chiral.
  \end{itemize}
  \end{thm}

\section{Fourth SWC} \label{hereisw4}
 
The fourth SWC is detected by its restriction to $S_8$. Here we have EA2Gs  $E=E_{\bf 1}$, $E_{(2,1,1)}$, and $E_{(2,2)}$ of rank $4$ and $E_{3}$ of rank $3$.
\subsection{Case $\ord(\pi) \geq 2$}
 
 For the rank $4$ EA2Gs $E$ we have
 \beq
 w_4^E(\pi)= \sum_e \frac{\deg \pi - \chi_\pi(e)}{8} d_e,
 \eeq
and by   Proposition \ref{elliptic.swc.here}, we have
\beq
w_4^{E_3}(\pi)=\frac{\deg \pi - \chi_\pi(\iota_{4})}{8} d_{E}.
\eeq
 
 \subsection{Case $\ord(\pi)=1$} 
 Put $\pi^+=\pi \oplus \rho_1^{a_1} \oplus \rho_2^{a_2}$ as in Proposition \ref{eps-rho}. 
Since $\ord(\pi^+) \geq 2$, the above formulas compute $w_4(\pi^+)$. Moreover,
 \beq
 w(\pi^+)=w(\pi)(1+a_1 w_1(\sgn)^2)(1+a_2w_2(\rho_2)+a_2 w_3(\rho_2)+a_2w_4(\rho_2)+ \cdots),
 \eeq
 and this gives
 \begin{equation} \label{w4.computed}
 w_4(\pi)=w_4(\pi^+)+w_2(\pi) (a_1 w_1(\sgn)^2 +a_2 w_2(\rho_2))+a_2w_4(\rho_2)+a_1a_2 w_1(\sgn)^2w_2(\rho_2).
 \end{equation}

 We now compute $w_4^E(\pi)$ for each EA2G.

\subsubsection{$E=E_{\bf 1}$} Let $\mc E_k$ be the elementary symmetric polynomials in the $v_i$.
We have $w_1^{\bf 1}(\sgn)=v_1+v_2+v_3+v_4=\mc E_1$, and $w^{\bf 1}(\pi_{\st})=(1+v_1)(1+v_2) \cdots (1+v_4)=1+ \mc E_1+ \mc E_2+\mc E_3 + \mc E_4$. Of course, $w^{\bf 1}(\rho_2)=(1+\mc E_1)(1+\mc E_1+ \cdots + \mc E_4)$.

 Hence \eqref{w4.computed} gives
 \beq
 w_4^{\bf 1}(\pi)=w_4^{\bf 1}(\pi^+)+w_2^{\bf 1}(\pi)((a_1+a_2)\mc E_1^2+a_2\mc E_2) \\
+ a_2 \mc E_4+a_2 \mc E_1 \mc E_3+ a_1a_2 \mc E_1^4 +a_1a_2 \mc E_1^2 \mc E_2. \\
    \eeq

\subsubsection{	$E=E_{3}$}
 Next, the restriction of $\sgn$ to $E=E_{3}$ is trivial, so $w_1^{E}(\sgn)=0$ and $w^{E}(\rho_2)=w^{E}(\pi_{\st})$. Moreover, the restriction of $\pi_{\st}$ to $E_{3}$ is the regular representation, so $w^{E}(\pi_{\st})$ is the Dickson product $1+d_E$ plus higher order terms. 
 Thus  \eqref{w4.computed} gives
 \beq
 w_4^{E}(\pi)=w_4^{E}(\pi^+)+a_2d_{E_{3}}.
 \eeq
 
 \subsubsection{$E=E_{(2,2)}$}
 The restriction of $\sgn$ to  $E_{(2,2)}$ is also trivial, so as before  $w_1^{(2,2)}(\sgn)=0$ and $w^{(2,2)}(\rho_2)=w^{(2,2)}(\pi_{\st})$.
Moreover, $\pi_{\st}$ restricted to $E_{(2,2)}$ is the product of the regular representations of the factors $E_{(2)}$; this gives  
 \beq
 w_2^{(2,2)}(\rho_2)=v_1^2+v_1v_2+v_2^2+v_3^2+v_3v_4+v_4^2
 \eeq
and
 \beq
 w_4^{(2,2)}(\rho_2)=(v_1^2+v_1v_2+v_2^2)(v_3^2+v_3v_4+v_4^2).
 \eeq
  Thus  \eqref{w4.computed} gives
  \beq
   w_4^{(2,2)}(\pi)=w_4^{(2,2)}(\pi^+)+a_2w_2^{(2,2)}(\pi)(v_1^2+v_1v_2+v_2^2+v_3^2+v_3v_4+v_4^2)+a_2 (v_1^2+v_1v_2+v_2^2)(v_3^2+v_3v_4+v_4^2).
   \eeq

 \subsubsection{$E=E_{(2,1,1)}$}
 
 Let $E'$ be the $E_2$ factor of $E$, and $E''$ the $E_1 \times E_1$ factor, so that $E=E' \oplus E''$.

  We have $w_1^E(\sgn)=v_3+v_4$ and $w^{E}(\pi_{\st}) = \mc D(E')(1+v_3)(1+v_4)$.
 So $w^E(\rho_2)= \mc D(E')(1+v_3)(1+v_4)(1+v_3+v_4)=\mc D(E') \mc D(E'')$.
  Whence $w^E_2(\rho_2)=d_1(E')+d_1(E'')$ and $w^E_4(\rho_2)=d_1(E')d_1(E'')$.
  (Recall that $d_1(E')=v_1^2+v_1v_2+v_2^2$ and $d_1(E'')=v_3^2+v_3v_4+v_4^2$.)

 Now  \eqref{w4.computed} gives
 \beq
 \begin{split}
 w_4^E(\pi) &=w_4^E(\pi^+)+w_2^E(\pi) (a_1 (v_3^2+v_4^2) +a_2 (d_1(E')+d_1(E'')))\\
 & +a_2 d_1(E')d_1(E'') +a_1a_2(v_3^2+v_4^2)(d_1(E')+d_1(E'')).\\
 \end{split}
\eeq

 \subsection{Chiral Case}
 For completeness, we address the case of $\pi$ chiral, i.e., that $\ord(\pi)=0$. Put $\pi'=\pi \oplus \sgn$; then $w_4(\pi)=w_4(\pi')+ w_1(\sgn) \cup w_3(\pi)$.
Since  $\ord(\pi') \geq 1$, we may compute $w_4(\pi')$ as in  the previous section. Moreover $w_3(\pi)$ was computed in Section \ref{sec:w3}. So we have in principle computed $w_4(\pi)$ in all cases.

\section{Defect of a Representation}

Given an orthogonal representation $\pi$ of a finite group $G$, we define the \emph{defect} $\delta(\pi)$ of $G$ as
\beq
\delta(\pi)=\max \{ i \mid \deg(\pi)-i<j \Rightarrow w_j(\pi)=0\}.
\eeq
In other words, when the defect is positive,  
\beq
0=w_{\topp}(\pi)= w_{\deg \pi-1}(\pi)=\cdots =w_{\deg(\pi)-\delta(\pi)+1}(\pi),
\eeq
 but $w_{\deg(\pi) - \delta(\pi)}(\pi) \neq 0$. Note that $0 \leq \delta(\pi) \leq \deg \pi$. We see that $w(\pi)=1$ iff $\delta(\pi)=\deg(\pi)$, and the top SWC $w_{\topp}(\pi) \neq 0$ iff $\delta(\pi)=0$.
 Note that $\delta(\pi) \geq \dim V^G$; this is an equality when $G$ is an elementary abelian $2$-group.
 If $H$ is a detecting subgroup for $G$, then also $\delta(\pi) \geq \dim V^H$.
This applies, for instance, when $H$ is a $2$-Sylow subgroup.

Now let $G$ be $S_n$ and $P_n$ a $2$-Sylow of $G$.   Law-Okitani \cite[Theorem C]{law2021plethysms} proved that \beq
\lim_{n\to\infty}\frac{\#\{\lambda\vdash n :  \left(V_\lambda \right)^{P_n}   \neq 0\}}{p(n)}= 1.
\eeq
When $V^{P_n} \neq 0$, then $\dim V^{P_n} \geq 1$, so $\delta(\pi) \neq 0$. Hence:
\begin{prop} \label{top.swc.usually}
\[
\lim_{n\to\infty}\frac{\#\{\lambda\vdash n : w_{\topp}(\pi_\lambda)=0\}}{p(n)}=1.
\]
\end{prop}
  
   For a positive integer $n$, put
 \beq
 n'=\begin{cases}
 n & \text{if $n$ is even} \\
 n-1 & \text{if $n$ is odd.} \\
 \end{cases}
 \eeq
 For $G=S_n$, we have 
 \beq
 \begin{split}
 \delta(\pi) &=\min(\delta(\pi|_{E_{\bf d}}) \mid 2^{\bf d}=n') \\
 &= \min(\dim V^{E_{\bf d}}  \mid 2^{\bf d}=n'). \\
 \end{split}
 \eeq
Let $\delta^{\bf d}=\delta(\pi|_{E_{\bf d}})$.
 
 \begin{example} 
Consider the standard representation $(\pi_{\st},V_{\st})$ of $S_n$ by permutation matrices.  
First suppose $n=2^k$, and let $E=E_k$ as in Section \ref{EA2Gs.sn}. Then $\dim V^{E_k}=1$.
Hence $\HH^{n-1}(S_n) \ni w_{n-1}(\pi_{\st}) \neq 0$.

For  $n$ even, let $E=E_{\bf d}$ with  $2^{\bf d}=n$. Let $r(\bf d)$ be the length of $\bf d$.
Then $\dim V_{\st}^{E_{\bf d}}=r(\bf d)$. Writing $\nu(n)$ for the number of $1$s in the binary expansion of $n$, we have $r({\bf d}) \geq \nu(n)$. 
It follows that $\delta^{\bf d}(\pi_{\st}) \geq \nu(n)$. Moreover if we take ${\bf d}_0$ to be the exponents of $2$ in the binary expansion of $n$, then $r({\bf d}_0)=\nu(n)$, hence
  $\delta(\pi_{\st})=\nu(n)$. For example, \(w_{n-1}(\pi_{\mathrm{st}})=0\)  when \(n\) is not a power of \(2\).
\end{example}

\bigskip

Some of this generalizes to hook partitions. Recall that $\pi_{\st}$ is the sum of $\pi_\la$ with the trivial one-dimensional representation, where $\la$ is the hook $(n-1,1)$.
 
Let $P_n$ be a $2$-Sylow of $S_n$. We write \(H(n)\) for the set of hook partitions of \(n\), so
\(H(n)=\{(n-x,1^x): 0\le x\le n-1\}\).
\begin{theorem}[Giannelli-Volpato]\label{thm:GV-hook-trivial}
Let  $\lambda=(n-x,1^x)\in H(n)$. Then
 \begin{equation} \label{GV.eqn}
 \dim   V_\la^{P_n} =  \binom{\nu(n)-1}{x},
  \end{equation}
hence $\delta(\pi_{\la}) \geq  \binom{\nu(n)-1}{x}$.
\end{theorem}

\begin{proof}
Equation \eqref{GV.eqn} is the specialization of \cite[Theorem~4.4]{GiannelliVolpato2024} and \cite[Theorem~3.2]{GiannelliVolpato2024} to the trivial
linear character, and the rest follows from earlier discussion. 
\end{proof}

  \section{Alternating Groups} \label{section: alternating}
   
We   conclude with  analogues to Theorem \ref{thm1} and Proposition \ref{top.swc.usually} for the alternating groups $A_n$. 
In this section,   all representations are \emph{complex}, unless otherwise specified. In particular $\Irr(A_n)$ will denote the irreducible complex representations of $A_n$.

\subsection{Orthogonal Representations}
Let $G$ be a finite group. Typically, some work is necessary to toggle between complex and real representations.   It is convenient, instead, to define SWCs of orthogonal representations. 
We say a representation $(\pi,V)$ of $G$ is \emph{orthogonal}, when there is a nondegenerate $G$-invariant symmetric bilinear form on $V$. Let $O\Pi(G)$ be the (equivalence classes of) orthogonal representations of $G$.

Orthogonal representations correspond neatly to real representations.
We say a complex representation $\pi$ is \emph{orthogonally irreducible}, provided $\pi$ is orthogonal, and $\pi$ does not decompose into a direct sum of two \emph{orthogonal} representations. Thus, an orthogonal representation $\pi$ is orthogonally irreducible iff one of the following holds:
\begin{enumerate}
\item 
$\pi$ is irreducible,
\item 
$\pi$ is of the form  $\varphi\oplus \varphi^{\vee}$ where $\varphi$ is irreducible but not orthogonal.
\end{enumerate}
We write `OIR' for ``orthogonally irreducible representation", and `IOR' for ``irreducible orthogonal representation'', meaning of the first type.

\subsection{Probability}
  For a finite subset $\mathcal S \subset O\Pi(G)$, 
we define $\epsilon_k(\mathcal S)$ to be the probability that $w_k(\pi) \neq 0$ given $\pi \in \mathcal S$. In other words,
\beq
\epsilon_k(\mathcal S)=\frac{ \#\{ \pi \in \mathcal S \mid w_k(\pi) \neq 0\}}{|\mathcal S|}.
\eeq

Theorem  \ref{thm1} says that $\epsilon_k(\Irr(S_n)) \to 0$ as $n \to \infty$. Similarly, let $\epsilon_{\topp}(\mc S)$ be the probability that $w_{\topp}(\pi) \neq 0$ given $\pi \in \mc S$.

We will use the following lemma a few times; its proof is elementary.
\begin{lemma} \label{alt.problemma} Suppose we have a sequence of groups $G_n$, and for each $n$ we have a nonempty finite subset $X_n \subset O\Pi(G_n)$, and a subset $Y_n \subseteq X_n$ with $|X_n| \sim |Y_n|$ as $n \to \infty$.
 Then as $n \to \infty$, we have $\epsilon_k(X_n) \to 0 \Leftrightarrow \epsilon_k(Y_n) \to 0$, and $\epsilon_{\topp}(X_n) \to 0 \Leftrightarrow \epsilon_{\topp}(Y_n) \to 0$.
\end{lemma}

\subsection{Representations of $A_n$}
We quickly recall material from the representation theory of $A_n$; a suitable reference is \cite[Section 4.6]{APB}.
By reflecting the Young diagram of a partition $\la \vdash n$ along the diagonal, we obtain another partition $\la' \vdash n$.
Let  $\Irr(S_n)^{SC}$ be the $\pi_\la$ with $\la = \la'$. Write $s(n) = \#\{ \la \vdash n \mid \la =\la'\}$ for its cardinality; 
From \cite[Section $5.1$]{markabacus} we know $s(n)$ is asymptotically the square root of $p(n)$, and in particular $s(n)/p(n) \to 0$ as $n \to \infty$.
Let $\Irr(S_n)^*$ be the $\pi_\la$ with $\la \neq \la'$. 
Restriction gives a map $\Irr(S_n)^* \to \Irr(A_n)$; if we write $\Irr(A_n)^{*}$ for its image, then $\Irr(S_n)^* \to \Irr(A_n)^*$ is $2$-to-$1$, the fibres being doubletons of the 
form $\{\la,\la'\}$.
The restriction of a $\pi \in \Irr(S_n)^{SC}$ to $A_n$ decomposes into two nonisomorphic irreducible constituents. The constituents are either both orthogonal, or both not orthogonal.
Say $\pi \in \Irr(S_n)^O$ when they are both  orthogonal, and $\pi \in \Irr(S_n)^X$ otherwise.
Induction gives a $2$-to-$1$-map $\Irr(A_n)^{SC} \to \Irr(S_n)^{SC}$, which restricts to a $2$-to-$1$-map $\Irr(A_n)^O \to \Irr(S_n)^O$.
Finally, $\IOR(A_n)=\Irr(A_n)^* \amalg \hspace{0.1cm}  \Irr(A_n)^O$, and $\OIR(A_n)=\IOR(A_n) \amalg \{\sigma \oplus \sigma^\vee \mid \sigma \in \Irr(A_n)^X\}$.

From these relationships we deduce that all of $|\Irr(A_n)|$, $|\Irr(A_n)^*|$, $|\IOR(A_n)|$, and $|\OIR(A_n)|$ are asymptotic to $\half p(n)$ as $n \to \infty$.

\begin{thm} Fix a positive integer $k$. As $n \to \infty$:

\begin{enumerate}
\item Both $\epsilon_k(\IOR(A_n))$ and $\epsilon_k(\OIR(A_n))$ tend to $0$. 
\item Both $\epsilon_{\topp}(\IOR(A_n))$ and $\epsilon_{\topp}(\OIR(A_n))$ tend to $0$.
\end{enumerate}
\end{thm}
\begin{proof}
Since $\epsilon_k(\Irr(S_n)) \to 0$ as $n \to \infty$, it follows that $\epsilon_k(\Irr(S_n)^*) \to 0$ as well by Lemma  \ref{alt.problemma}.
If $w_k(\pi)=0$ for $\pi \in \Irr(S_n)^*$, then also $w_k(\pi|_{A_n})=0$. Therefore
\beq
\begin{split}
\epsilon_k(\Irr(A_n)^*) &= \frac{ \#\{ \sigma \in \Irr(A_n)^* \mid w_k(\sigma) \neq 0\}}{|\Irr(A_n)^*|} \\
                &= \frac{  \#\{ \pi \in \Irr(S_n)^* \mid w_k(\pi|_{A_n}) \neq 0\}}{|\Irr(S_n)^*|} \\
                & \leq \epsilon_k(\Irr(S_n)^*).\\
                \end{split}
                \eeq

By the asymptotics mentioned above, we deduce the first statement from Lemma \ref{alt.problemma}, and the second statement is similar.
\end{proof}

\bibliographystyle{alpha}
\bibliography{mybib} 
   
 \end{document}